\documentclass[a4paper]{amsart}

\usepackage{epsfig}
\usepackage{amsthm,amsfonts}
\usepackage{amssymb,graphicx,color}
\usepackage[all]{xy}
\usepackage{verbatim}
\usepackage{hyperref}
\usepackage{enumitem}

\usepackage{url}

\newtheorem{theorem}{Theorem}[section]
\newtheorem*{thA}{Theorem A}
\newtheorem*{thB}{Theorem B}
\newtheorem*{thC}{Theorem C}
\newtheorem{lemma}[theorem]{Lemma}
\newtheorem{corollary}[theorem]{Corollary}

\newtheorem{conjecture}{Conjecture}
\newtheorem{example}[theorem]{Example}

\newtheorem{problem}[theorem]{Problem}

\begin{document}

\title[Foliations transverse to a closed conformal vector field]{Foliations transverse to a closed conformal vector field}

\author[E. da Silva]{Euripedes da Silva}
\address{E. da Silva: Departamento de Matem\'atica, Instituto Federal de Educa\c{c}\~ao, Ci\^encia e Tecnologia do Cear\'a,
	Av. Parque Central, 1315, Distrito Industrial I, 61939-140, 
	Maracana\'u-CE, Brazil. \newline 
	and 	      Departamento de Matem\'atica, Universidade Federal do Cear\'a,
	Rua Campus do Pici, s/n, Bloco 914, Pici, 60440-900, 
	Fortaleza-CE, Brazil.
	E-mail: {\tt euripedes.carvalho@ifce.edu.br}
} 

\author[I. Gonçalves]{Ícaro Gonçalves}
\address[I. Gonçalves]{Centro de Matematica, Computacção e Cognição, Universidade Federal do ABC, Av. dos Estados, 5001, Santo André-SP, 09280-560, Brasil.}
\email{icaro.goncalves@ufabc.edu.br}  

\author[J. Pereira]{Júlio Pereira}
\address[J.Pereira]{Instituto Federal de Educação, Ciência e Tecnologia Baiano, Campus Valença, Rua Glicério Tavares, s/n, Valença-BA, 45400-000, Brasil.}
\email{julio.pereira@ifbaiano.edu.br}

\thanks{The first-named author was partially supported by the CNPq-Brazil grant 405618/2023-9 and by the Serrapilheira Institute (grant number Serra -- R-2110-39576).}
\thanks{The third author was financed by Coordenação de Aperfeiçoamento de Pessoal de Nível Superior- Brazil (CAPES) - Finance Code 001.}

\begin{abstract}
	In this article, we study the geometric properties of codimension one foliations on Riemannian manifolds equipped with vector fields that are closed and conformal. Apart from its singularities, these vector fields define codimension one foliations with nice geometric features, which we call Montiel Foliations. We investigate conditions for which a foliation transverse to this structure has totally geodesic leaves, as well as how the ambient space and the geometry of the leaves forces a given foliation into a Montiel Foliation. Our main results concern minimal leaves and constant mean curvature foliations, having compact or complete noncompact leaves. Finally, we characterize totally geodesic foliations by means of its relation to a prior Montiel Foliation. 
\end{abstract}
\maketitle


\section{Introduction}

The study of the extrinsic geometry of codimension one foliations on Riemannian manifolds has been undertaken by numerous authors. In the case of totally geodesic foliations, from existence to the topology and geometry of manifolds carrying such foliations, we refer to  [\cite{abe:1973}, \cite{brito}, \cite{BritoW-1988}, \cite{BG}, \cite{Ferus:1970}, \cite{Ghys-1983}, \cite{JohnsonW:1980}, \cite{SilvaSR:2023}]. When the leaves have constant mean curvature, including the minimal case, several authors have discussed rigidity results, see for instance [\cite{BarbosaGS-1987}, \cite{BarbosaKO:1991}, \cite{Oshikiri-1981}, \cite{Oshikiri-1982}, \cite{Oshikiri-1987}]. We also refer to \cite{czar} and the book \cite{rov-pawel} and the references therein.

In 2008, Meeks III, Pérez, and Ros proposed the following conjecture in \cite[Conjecture 5.1.2]{MeeksPR:2008}:

\begin{conjecture}\label{conjecture:control_mean_curv_general}
	Let $\mathcal{F}$ be a codimension one foliation of a complete Riemannian manifold $\overline{M}^{n+1}$. Assume that there is $K_0 \geq 0$ such that $\overline{M}$ has Ricci curvature bounded from below by $-nK_0$. Suppose that each leaf $L$ of $\mathcal{F}$ has constant mean curvature $H_L$. Then $|H_L| \leq \sqrt{K_0}$.
\end{conjecture}

Sampaio and Silva provided a positive answer to this conjecture in \cite{SampaioS:2024} (see Theorem 4.6 in \cite{SampaioS:2024}) when $\overline{M}$ is a compact manifold. In this case, they prove even more: If there is a codimension one CMC foliation of a compact Riemannian manifold such that there is $K_0 \geq 0$ and the Ricci curvature of $M$ is bounded from below by $-nK_0$, then for each leaf $L$ of $\mathcal{F}$, the mean curvature $H_L$ of $L$ satisfies $|H_L| \leq \sqrt{K_0}$. Furthermore, if the foliation contains a leaf $L$ whose absolute mean curvature is $|H_L| = \sqrt{K_0}$, then either $K_0 = 0$ and all the leaves of $\mathcal{F}$ are totally geodesic, or $K_0 > 0$ and there is a totally umbilical leaf in $\mathcal{F}$.

In 1986, Solomon \cite{Solomon-1986} posed the following problem:
\begin{problem}\label{problem:minimal_solomon}
	Is every foliation of $\mathbb{R}^{n+1}$, $n\leq 7$, by minimal hypersurfaces a foliation by parallel affine hyperplanes? 
\end{problem}

If the leaves are properly embedded, Problem \ref{problem:minimal_solomon} has an affirmative answer if \(n+1 < 8\) (see \cite[Theorems 1 and 3]{ShoenS:1981}), and a negative answer if \(n+1 > 8\), since in those dimensions there are minimal graphs that are not hyperplanes. In the critical dimension \(n+1 = 8\), Solomon noted in \cite{Solomon-1986} that resolving this question appears quite difficult, even when the leaves are properly embedded. The problem becomes much more complicated when the leaves are not required to be properly embedded.

In 1999, Montiel \cite{montielfoliation:1999} classified hypersurfaces with constant mean curvature on Riemannian manifolds carrying a vector field $\xi$ that is closed and conformal. This type of vector field yields singular Riemannian foliations with totally umbilical leaves on its regular stratum. We call these foliations Montiel Foliations, or simply $\mathcal{F}(\xi)$.

On the other hand, Brasil et al. \cite{BrasilDSS:2022} explored the idea of foliations transverse to a closed conformal vector field at infinity. They obtained a Bernstein-type result for foliations transverse to this structure at infinity. 

Based on the aforementioned discussion, we pose the following question: {\it Given a codimension one foliation transverse to a closed conformal vector field on a Riemannian manifold, what is the relation between the geometry of the leaves and the foliation associated with this vector field?} Our aim is to provide partial answers to these questions and conjectures for transversaly foliation to a conformal and closed vector field. When the foliation has at least one minimal leaf, we have the following:

\begin{thA}\label{theorem:1} Let $\overline{M}^{n+1},$ $n\geq1,$ be a complete Riemannian manifold with non-positive Ricci curvature equipped with an unit vector field $\xi$ that is closed and conformal. Let $\mathcal{F}$ be a codimension one foliation of $\overline{M}$ transverse to field $\xi$. Given a minimal leaf $L$ of $\mathcal{F}$, suppose the support function $\nu=\overline{g}(\xi,N)$ does not change sign on $L.$ Then
	\begin{itemize}
		\item[1.] If $L$ is compact, then $L$ is totally geodesic.
		\item[2.] If $L$ is complete non-compact and $A(\xi^{\top})\in \mathcal{L}^1(\mathcal{F})$, then $L$ is a totally geodesic. In addition, if the inequality
		\begin{eqnarray}\label{desigualdadeNU}
			\nu\geq\dfrac{1}{\|A\|^{2}+1}
		\end{eqnarray}
		holds on $L\in\mathcal{F},$ then $L\in \mathcal{F}(\xi)$.
		\item[3.] Assume $L$ is a complete non-compact leaf with polynomial volume growth. If the Weingarten operator $A$ is bounded and \eqref{desigualdadeNU} holds, then $L$ is totally geodesic and $L\in \mathcal{F}(\xi)$.
	\end{itemize}
\end{thA}

Czarnecki \cite{czar} proved that in an Hadamard manifold with sectional curvature $k$, with $k\leq -a^2$ and $a>0$, a foliation such that the norm of the second fundamental form of its leaves is bounded by $a$ must be Hadamard itself, meaning all of its leaves are Hadamard manifolds. The proof of Theorem A, specifically the proof of its first assertion, makes clear the presence of planes with zero sectional curvatures.   

If we have a foliation by constant mean curvature hypersurfaces, we have a rigidity result for compact manifolds having non-positive Ricci curvature. 

\begin{thB}
Let $\overline{M}^{n+1}$ be a compact, orientable Riemannian manifold with non-positive Ricci curvature equipped with a Montiel Foliation $\mathcal{F}(\xi)$. Let $\mathcal{F}$ be a CMC foliation of $\overline{M}^{n+1}$ transverse to $\xi$ and $L\in\mathcal{F}$. Assume that the support function $\nu=\overline{g}(\xi,N)$ does not change sign.
\begin{itemize}
	\item[1.] If $L$ is compact, then $L$ is totally geodesic.
	\item[2.] If $L$ is complete non-compact and the vector field $X=A(\xi^{\top})-\mathcal{H}\xi^{\top}$ is Lebesgue integrable on  $L,$ then $L$ is totally geodesic.
\end{itemize}
In particular, if $\nu\geq 1/\left( \|A\|^{2}+1\right),$ then $\mathcal{F}$ coincides with $\mathcal{F}(\xi)$.
\begin{itemize}
	\item[3.] Suppose $L$ is a complete non-compact leaf  with polynomial volume growth, $\|A(\xi^{\top})\|^2\geq \mathcal{H} \overline{g}( A(\xi^{\top}),\xi^{\top})$ and $\nu\geq \frac{1}{\|A\|^{2}+1}$ on $L$. If the Weingarten operator $A$ is bounded, then $L\in \mathcal{F}(\xi)$.
\end{itemize}
\end{thB}

Finally, we define a metric invariant in order to characterize totally geodesic leaves that are transverse to a closed conformal vector field on
Riemannian manifolds. More precisely:
$$\mathcal{G}_{L}^{\xi} := \sup_{p\in L} \left\{\mbox{div}_{L}\left(A(\xi^{\top})-\mathcal{H}\xi^{\top} \right) + \mbox{Ric}_{\overline{M}}\left(\xi^{\top},N\right) \right\}.$$

There are similar results in the literature, for instance see \cite{ChavesS-2020}, \cite{Gomes-2004} and \cite{SilvaSR:2023}.

\begin{thC}\label{TeoInvariante}
	Let $\overline{M}^{n+1}$, $ n\geq1$, be a Riemannian manifold equipped with Montiel foliation $\mathcal{F}(\xi)$. Let $\mathcal{F}$ be a codimension one foliation transverse to $\xi$ and $L\in\mathcal{F}$ a complete leaf. If $\nu\geq\nu_{0}>0$ and $\dim L\geq \frac{1}{\nu_{0}},$ then
	\begin{eqnarray} 
		\label{eqGmaiorqueH}
		\mathcal{H}_L^2\leq \mathcal{G}_{L}^{\xi}\,\,
	\end{eqnarray} and
	\begin{equation*}
		L \,\text{is totally geodesic, if and only if,}\ \mathcal{G}_{L}^{\xi}=0.	
	\end{equation*}

\end{thC}


%

\section{Preliminaries}\label{preliminaries}
In this section, we introduce some basics facts, definitions and notations.

\subsection{Riemannian Manifolds and Foliations}\label{Riemannian Manifolds}
Let $\overline{M}$ be an Riemannian manifold with dimension $n+1$, $n\geq1$, with the Riemannian metric denoted by $\overline{g}$. 

Let $\mathcal{F}$ be an oriented and transversely oriented codimension one foliation on $\overline{M}$. We can choose an orthonormal adapted frame $\{ e_{1}, \,\hdots\,, e_{n}, e_{n+1}\}$ around any $p\in \overline{M}$ where $e_{1}, \,\hdots\,, e_{n}$ are tangent to $\mathcal{F}$ and $e_{n+1}$ is normal to the leaves of $\mathcal{F}.$ There is a convention of indices, $1\leq A,B,C,\, \hdots \leq n+1$ and $1\leq i, j, k,\, \hdots \leq n.$ 

We say that $\mathcal{F}$ is a transverse foliation to a submanifold $S\subset \overline{M}$ if for each leaf $L\in\mathcal{F}$ and for each point $p\in L\cap S,$ we have $$T_{p}L+T_{p}S=T_{p}\overline{M}.$$   

In addition, we say that the foliations $\mathcal{F}_{1}$ and $\mathcal{F}_{2}$ are transverse if $\mathcal{F}_{1}$ is transverse to each one of the leaves $L\in \mathcal{F}_{2}.$

We define the second fundamental form $\mathcal{B}$ of the leaf  $L\in \mathcal{F}$ in the direction of $e_{n+1}=N$ by $\mathcal{B}(e_{i},e_{j})=\overline{g}\left(-\overline{\nabla}_{e_{i}}N,e_{j} \right)$, with associated Weingarten operator $A(e_{i})=-\left(\overline{\nabla}_{e_{i}}N\right)^{\top}.$

The norm of second fundamental form $\mathcal{B}$ is $$\|\mathcal{B}\|^{2}=\left( \sum_{i,j}\overline{g}\left( -\overline{\nabla}_{}e_{i}N, e_{j} \right)^{2} \right)^{\frac{1}{2}}.$$ 

The mean curvature vector field is defined by $$\vec{H}=\dfrac{1}{n}\sum_{i}\left(\overline{\nabla}_{e_{i}}e_{i}\right)^{\perp},$$ 
and the mean curvature function corresponding to $N$ is $\mathcal{H}=\frac{1}{n}tr(A)$.

Given smooth vector fields $X$ and $Y$, the Ricci curvature in the directions of $X$ and $Y,$ is
$$\mbox{Ric}_{\overline{M}}(X,Y)=\sum_{A}\overline{g}\left( \overline{R}(e_{A},X)Y,e_{A} \right).$$ 

For smooth vector field $X$ in $T\overline{M},$ we can define the divergent of the leaf  $L\in\mathcal{F}$ and the divergent of $\overline{M}$, respectively
$$\mbox{div}_{L}(X)=\sum_{i}\overline{g}\left(\overline{\nabla}_{e_{i}}X,e_{i}\right)$$ and 
$$\mbox{div}(X)=\mbox{div}_{L}(X)+\overline{g}\left( \overline{\nabla}_{N}N,N \right).$$

Let $L$ be a leaf of the foliation $\mathcal{F}$. We will say $L$ is minimal, totally geodesic or totally umbilic, if $\mathcal{H}_{L}=0,$ $\mathcal{B}\equiv0$ or $\mathcal{B}(X,Y)=\lambda\overline{g}\left(X,Y\right),$ for all $X,Y\in TL$ and $\lambda\in C^{\infty}(\overline{M})$, respectively.  The foliation $\mathcal{F}$ is called minimal, totally geodesic or totally umbilic if each leaf  $L\in\mathcal{F}$ is, respectively, minimal, totally geodesic or totally umbilic.

We say that $\|X\| \in \mathcal{L}^1(\mathcal{F})$ if and only if $\|X^{\top}\| \in \mathcal{L}^1(L)$ for each leaf $L$ of the foliation $\mathcal{F}$, where $\mathcal{L}^1(L)$ is the space of Lebesgue integrable functions, and $X^{\top}$ represents the tangent projection of the field $X$ onto the leaf $L$.

\subsection{Closed conformal fields}\label{closed conformal fields}

Suppose that $\overline{M}$ is equipped with a closed and conformal vector field $\xi$, that is, there is a $\varphi\in {C^{\infty}(\overline{M})}$ such that 
\begin{equation}
	\overline{\nabla}_{X}\xi=\varphi X,\,\,\,\forall X\in T\overline{M},
\end{equation}
where $\overline{\nabla}$ is a Levi-Civita connection on $\left(\overline{M},\overline{g}\right).$ 

Proposition 1 of \cite{montielfoliation:1999} exhibits some properties of manifolds carrying a non-trivial $\xi$:
\begin{itemize}
	\item[(a)] The set of points $p\in\overline{M}$ such that $\xi_{p}=0,$ denoted for $Z(\xi),$ is a discret set;
	 \item[(b)] The unit vector $\xi_{0}=\dfrac{\xi}{\|\xi\|}$ defined on open dense set $\widetilde{M}=\overline{M}\setminus Z(\xi)$ is such that
	 $$\overline{\nabla}_{\xi_{0}}\xi_{0}=0,\,\, \text{and if}\,\, \overline{g}(X,\xi_{0})=0,\,\, \text{then}\,\,\overline{\nabla}_{X}\xi_{0}=\dfrac{\varphi}{\|\xi_{0}\|}X.$$ In particular, the flow of $\xi$ is a geodesic flow; and,
	 \item[(c)] The $n-$dimensional distribution $\mathcal{D}$ defined by
	 $$p\in\overline{M} \mapsto \mathcal{D}(p)=\left\{ X\in T_{p}\overline{M}\,|\,\overline{g}(\xi_{p},X)=0 \right\}$$ is integrable to a codimension one umbilical Riemannian foliation $\mathcal{F}(\xi)$, oriented by $\xi$. In addiction, the functions $\|\xi\|$, $\mbox{div}(\xi)$ and $\xi(\varphi)$ are constants on connected leaves of $\mathcal{F}(\xi)$ and each leaf $L$ has constant mean curvature $\mathcal{H}_{L}=-\dfrac{\mbox{div}(\xi)}{(n+1)\|\xi\|}.$
\end{itemize} 

We call $\mathcal{F}(\xi)$ a \textit{Montiel Foliation generated by the vector field $\xi$} or simply \textit{Montiel Foliation}. Montiel classified all complete Riemannian manifolds equipped with a non-trivial closed conformal vector field $\xi$. We refer to \cite{montielfoliation:1999} for a detailed discussion. 


\subsection{Volume growth}

Consider $\overline{M}$ an oriented complete (non-compact) and connected Riemannian manifold and $B(p,t)$ the geodesic ball centered at $p\in\overline{M}$ with radius $t$.

Let $g$ and $f$ be real functions defined on $(a,+\infty)$, $a\in \mathbb{R}$, with $g$ strictly positive. If there are constants $x_{0}\in (a,+\infty)$ and $M>0$ such that $\|f(x)\|\leq M g(x),$ for all $x\geq x_0$,  we write $f=\mathcal{O}(g)$.   

We say that $\overline{M}$ has polynomial volume growth if there is a point $p\in\overline{M}$ such that 
$$\mbox{vol}\left( B(p,t) \right)=\mathcal{O}\left(P(t)\right)$$ 
when $t\rightarrow +\infty$, where $\mbox{vol}$ denote the Riemannian volume and $P>0$ is a polynomial function.

\section{Montiel and Transverse Foliations}

The next lemma was proved in \cite{BrasilDSS:2022}. For reader's comfort, we are going to present a geometric and direct proof of this result.
\begin{lemma}\label{prop:eqfundamental}
	Let $\overline{M}^{n+1}$ an oriented connected Riemannian manifold equipped with Montiel Foliation $\mathcal{F}(\xi)$.  Let $\mathcal{F}$ be a codimention one foliation transverse to field $\xi.$ Then, for all  $L\in\mathcal{F},$
	\begin{equation}\label{equacaofundamental}
		\mbox{div}_{L}\left( A\left(\xi^{\top}\right)\right)=-\mbox{Ric}_{\overline{M}}(\xi^{\top},N)+\xi^{\top}(n\mathcal{H})+\nu\|A\|^{2}+n\varphi \mathcal{H},
	\end{equation}
where $A$ is the Weingarten operator of $L$, $\mathcal{H}$ is the mean curvature function of the leaf $L$, $\nu$ is the support function $\nu=\overline{g}(\xi,N)$ and $\xi^{\top}$ denotes the orthogonal projection of $\xi$ on $TL$.
\end{lemma}
\begin{proof}
	Let $p\in L$ and $\{e_{i}\}_{i=1}^{n}$ be an orthonormal geodesic local frame on a neighborhood of $p$, $\left(\overline{\nabla}_{\xi^{\top}} e_{i}\right)=\nabla_{\xi^{\top}}e_{i}\equiv0$, where $\overline{\nabla}$ and $\nabla$ are the Levi-Civita connections of $\overline{M}$ and $L$, respectively. Expanding the leaf divergence of $A\left(\xi^{\top}\right)$,
	\begin{eqnarray}\label{eqfundamental-E1}
	\mbox{div}_{L}\left( A\left(\xi^{\top}\right)\right)= -\mbox{Ric}_{\overline{M}}(\xi^{\top},N)+\xi^{\top}\left(n\mathcal{H}\right)+\sum^{n}_{i=1}\overline{g}\left( \nabla_{[\xi^{\top},e_{i}]}N,e_{i} \right).\nonumber
	\end{eqnarray}

Since $\mathcal{F}$ is integrable, $[\xi^{\top},e_{i}]\in T\mathcal{F}$, for all $i=1,\hdots, n.$ The symmetry of the second fundamental form implies that $\sum^{n}_{i=1}\overline{g}\left( \nabla_{[\xi^{\top},e_{i}]}N,e_{i} \right)=\varphi(n\mathcal{H})+\nu\|A\|^{2}$, finishing the proof. 
\end{proof}
 
Equation $\eqref{equacaofundamental}$ allows us to relate geometric aspects of an arbitrary codimension one transverse foliation $\mathcal{F}$ and the closed conformal vector field $\xi$.

\section{Proof of main results}

\subsection{Proof of Theorem A}
\begin{proof}
	
	Let $L\in \mathcal{F}$. From the Lemma $\ref{prop:eqfundamental}$, 
	\begin{eqnarray*}
		\mbox{div}_{L}\left(A(\xi^{\top})\right)=-\mbox{Ric}_{\overline{M}}(\xi^{\top},N)+\xi^{\top}(n\mathcal{H})+\nu\|A\|^{2}+n\varphi \mathcal{H}.
	\end{eqnarray*} 
If $L$ is minimal and $\nu$ does not change sign, then we can write the following inequality 
	\begin{eqnarray}\label{eq8}
	\mbox{div}_{L}\left(A(\xi^{\top})\right)\geq\nu\|A\|^{2}\geq0.
\end{eqnarray}
In addiction, as a direct consequence of the Cauchy-Schwarz inequality,
	\begin{eqnarray}\label{eq9}
	\nu=\overline{g}(\xi,N)\leq |\overline{g}(\xi,N)|\leq \|\xi\| \|N\|\leq 1  \iff
	0\leq 1-\nu.
\end{eqnarray} 


	{\it Case 1}: If $L$ is a compact leaf, Stoke's Theorem together with inequality $\eqref{eq8}$ give
\begin{equation}\label{eq10}
	0=\int_{L}\mbox{div}_{L}(A(\xi^{\top}))dL\geq\int_{L}\nu\|A\|^{2}dL\geq 0.
\end{equation}

	Thus $\nu\|A\|^{2}\equiv0,$ and since the foliation $\mathcal{F}$ is transverse to the field $\xi$, we have $\|A\|\equiv0$ along the leaf $L,$.

{\it Case 2}: If $L$ is a complete non-compact leaf, then inequality \eqref{eq8} implies that $\mbox{div}_{L}(A(\xi^{\top}))$ does not change sign and since $\|A(\xi^{\top})\|\in \mathcal{L}^{1}(\mathcal{F})$, Proposition $1$ of \cite{CamargoCS-2010}(see also the main Theorem in \cite{Karp:1981}) allows us to conclude that $\mbox{div}_{L}(A(\xi^{\top}))=0.$ Thus, $\|A\|\equiv0$ along the leaf $L$.

In both cases, $\mbox{Ric}_{\overline{M}}(\xi^{\top},N)\equiv0$ along the leaf $L.$ In particular, if the inequality $\eqref{desigualdadeNU}$ is satisfied, then
$$\nu\|A\|^{2}\geq \dfrac{\|A\|^{2}}{\|A\|^{2}+1}\geq 1-\nu\geq0.$$ Since $L$ is totally geodesic we can conclude that $\nu=1$ along of leaf  $L.$ This means that along the leaf  $L$ the vectors fields $\xi$ and $N$ are the same, i. e., $L$ is a leaf of $\mathcal{F}(\xi).$

{\it Case 3}: Define $f=1-\overline{g}(N,\xi)$. From $\eqref{eq9}$, $f\geq 0$ on $L$. For instance, if $f\equiv 0$, then $N$ coincides with $\xi$. In this case, given $u, v\in T_pL$ we have 
\begin{eqnarray*}
	\overline{g}(A(u),v)  =  -\overline{g} (\overline{\nabla}_uN,v)
	 =  -\overline{g}(\overline{\nabla}_u\xi,v)
	 =  \overline{g}(-\varphi u,v),
\end{eqnarray*}
therefore, $A\equiv-\varphi Id$, proving that $L$ is totally umbilic. On the other hand,  $\mathcal{F}$ is minimal, which means $L$ is totally geodesic. 

We are going to prove that case 3 implies $f\equiv 0$.

Let $p\in L\in \mathcal{F}$ and $v\in T_pL$. Then
\begin{eqnarray*}
	\overline{g}(\nabla \nu,v)  =  \overline{g}(\nabla_v N,\xi)+\overline{g}(N,\nabla_{v}\xi)
	 =  \overline{g}(-A(\xi^{\top}),v),
\end{eqnarray*}
hence, $\nabla \nu =-A(\xi^{\top})$ and $\overline{g}(\nabla \nu,A(\xi^{\top}))=-\|A(\xi^{\top})\|^2.$  From the fact that $\nabla \nu=-\nabla f$, $\overline{g}(\nabla f, A(\xi^{\top}))=\|A(\xi^{\top})\|^2\geq 0.$
By the Equation \eqref{equacaofundamental}, we get  
$$\mbox{div}_L(A(\xi^{\top})) =-\mbox{Ric}_{\overline{M}}(\xi^{\top},N)+\nu\|A\|^2.$$
Since $\overline{M}$ has non-positive Ricci curvature, $\mbox{div}_{L}(A(\xi^{\top}))\geq\nu \|A\|^2.$
By inequality \eqref{desigualdadeNU}, we have $\mbox{div}_L(A(\xi^{\top}))\geq f$ along $L$. The Weingarten operator $A$ is bounded and the field $\xi$ is unitary, then the field $A(\xi^{\top})$ is also bounded. Since $L$ has polynomial volume growth, by Theorem $2.1$ of \cite{AliasCN-2021}, we can conclude that $f\leq 0$ along $L$. And, by inequality \eqref{eq9}, we have $f\geq 0$ on $L$, so $f\equiv 0$ on $L$.
\end{proof} 
 
Regarding the conclusion that the leaf $L$ belongs to the Montiel foliation, we emphasize that $\nu\geq \frac{1}{\|A\|^{2}+1}$ and assuming $\nu$ does not change sign are absolutely necessary in Theorem A. The next couple of examples illustrate this fact.  

\begin{example}\label{exemplo1}
	 Let $\overline{M}^{3}$ be the tridimensional Euclidean space equipped with the conformal vector field $\xi=(0,0,1)$ and $\mathcal{F}(\xi)$ the Montiel foliation with respect to $\xi$. Let $\mathcal{F}=\{L_{\alpha}\}$ be a foliation given by complete non-compact planes
	 $$L_{\alpha}=\left\{  (x,y,z)\in\mathbb{R}^{3} | P_{\alpha}(x,y,z)=0, \text{where}\ P_{\alpha}=z-(x+y+a)  \right\}.$$ The support function $\nu=\overline{g}(\xi,N)=\frac{1}{\sqrt{3}}$ does not change sign, but $\overline{g}(\xi,N)<\frac{1}{\|A\|^2+1}=1$; see Figure~\ref{fig:exemplomin}.

\end{example} 

	\begin{figure}[!htb]
	\centering
	\includegraphics[scale=0.2]{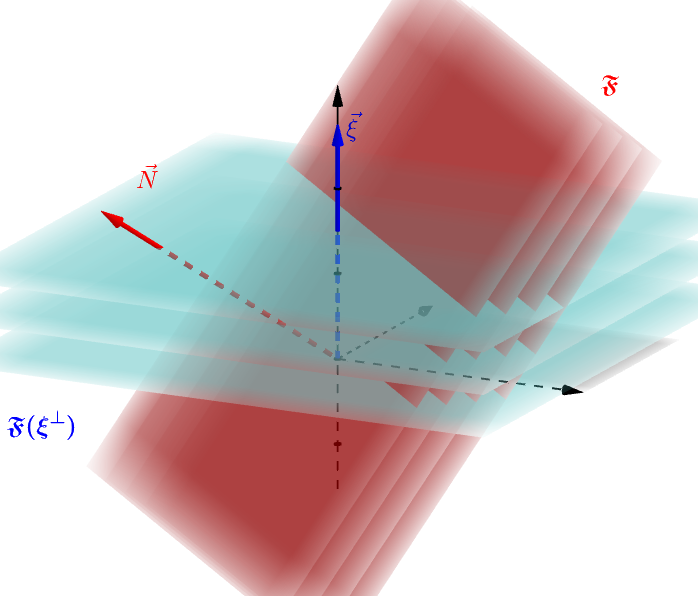}
	\caption {Counter-example: when $\nu$ does not satisfy inequality $\eqref{desigualdadeNU}$}
	\label{fig:exemplomin}
\end{figure}

\begin{example}\label{exemplo2}	 Let $\overline{M}^{3}$ be the tridimensional Euclidean space equipped with the conformal vector field $\xi=(1,0,0)$ and $\mathcal{F}(\xi)$ the Montiel foliation associated to $\xi$. For each $\beta\in\mathbb{R},$ consider the complete non-compact parameterized surfaces given by $S_{\beta}(u,v)=(u,u+\beta,v)$ and foliation  $\mathcal{F}=\{S_{\beta}\left(\mathbb{R}^2\right): \beta\in \mathbb{R}\}$. $\mathcal{F}$ is oriented by the unit vector field  $$N=\dfrac{1}{\sqrt{1+9u^{4}}}(3u^{2},-1,0).$$ In addition, the principal curvatures of the leaf  $L_{\beta}$ are $\kappa_{1}=0$ and $\kappa_{2}=\frac{6u}{(1+9u^{4})^{3/2}}.$ Hence, the foliations $\mathcal{F}$ and $\mathcal{F}(\xi)$ are transversal foliations such that $\nu$ changes sign and $\mathcal{F}$ is not a totally geodesic foliation. 
\end{example}
	\begin{figure}[!htb]
	\centering
	\includegraphics[width=0.8\linewidth]{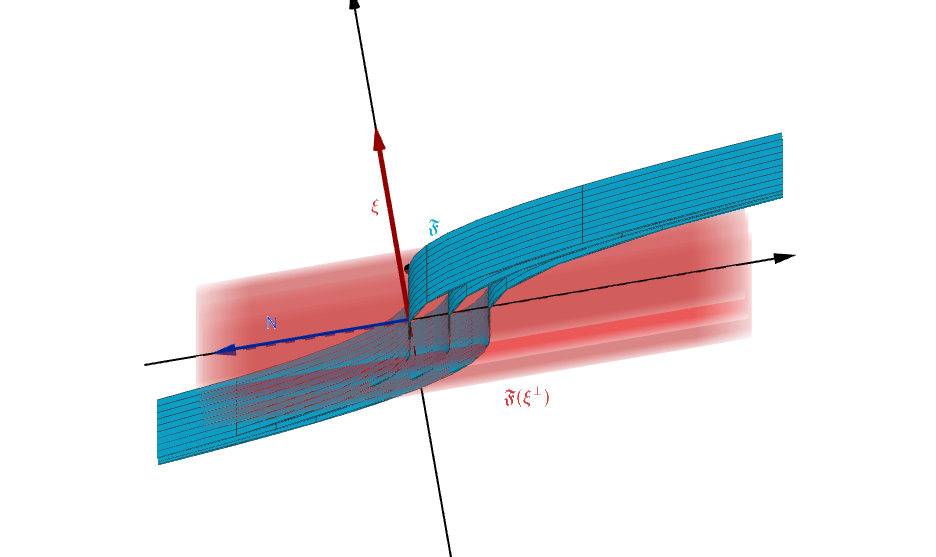}
	\caption{Counter-exemple: when $\nu$ changes sign}
	\label{figura:exemploAldirEuripedes}
\end{figure}

In the last part of Theorem A, the hypothesis that $A(\xi^{\top})\in \mathcal{L}^1(\mathcal{F})$ cannot be removed. 

\begin{example}
	By \cite{BombieriDG-1969}, there is a smooth function $f\colon \mathbb{R}^n\to \mathbb{R}$ $(n\geq 8)$ such that $Graph(f)$ is a minimal hypersurface of $\mathbb{R}^{n+1}$ that is not a hyperplane. Then, we define the foliation $\mathfrak{F}:= \bigcup_{t\in \mathbb{R}}{L_{t}}$ on $\mathbb{R}^{n+1}$, where each leaf $L_t$ is given by
	$$	
	L_t=\{(x,y)\in\mathbb{R}^n\times \mathbb{R}; y=f(x)+t\}.
	$$
	We have that each leaf $L_t$ is minimal hypersurface.
\end{example}
For a caracterization of minimal hypersurface on an Euclidean space see Theorem 4.1 \cite{SampiaoS2:2024}.

Changing the ambient space $\overline{M}$ from non-positive Ricci curvature to an Einstein manifold in the Theorem A, an analogous argument implies an alternative version of Theorem $3.1$ in \cite{BrasilDSS:2022}. 
\begin{corollary}
	Let $\overline{M}$ be an Einstein manifold equipped with a unit vector field $\xi$ which is closed and conformal. Let $\mathcal{F}$ be a minimal and transverse foliation with respect to field $\xi.$ If the support function $\nu=\overline{g}(\xi,N)$ satisfies the inequality $\ref{desigualdadeNU}$, then $\mathcal{F}$ is the Montiel foliation. Consequently, $\mathcal{F}$ is a totally geodesic foliation.
\end{corollary}

\begin{proof}
	Notice that $\mbox{Ric}_{\overline{M}}(\xi^{\top},N)=0.$ By the minimality of $\mathcal{F}$,
	$$\mbox{div}_{L}(A(\xi^{\top}))=\nu ||A||\geq 1-\nu.$$ So the Cauchy-Schuwarz inequality implies $1-\nu\geq0$. Then, by a similar argument in Theorem A, we conclude that $\mathcal{F}$ is a totally geodesic foliation. 
\end{proof}


\subsection{Proof of Theorem B}
\begin{proof}
	The mean curvature function $\mathcal{H}\colon \overline{M}\rightarrow \mathbb{R}$ assigns each point $p\in \overline{M}^{n+1}$ to the mean curvature of the leaf passing through $p$. Since $\mathcal{F}$ is a CMC foliation, Proposition $2.31$ of \cite{BarbosaKO:1991} asserts that $\mathcal{H}$ is constant on $\overline{M}$ or the set of maximum points of $\mathcal{H}$ contains at least one compact leaf. 
	
	{\it Claim}: $\mathcal{H}$ is constant on $\overline{M}$.
	
	Suppose that $\mathcal{H}$ is not constant on $\overline{M}.$ There exists a compact leaf  $\Sigma\in\mathcal{F}$ where $\mathcal{H}$ attains maximum value $\mathcal{H}_{\Sigma}$. Taking the restriction of $n\mathcal{H}\varphi$ to  $\Sigma$, the fact that $(n+1)\varphi=\mbox{div}(\xi)$ and applying Stokes' theorem, we have
	
	\begin{eqnarray}\label{eq11}
		\int_{\Sigma}{n\mathcal{H}\varphi }d\Sigma   &=&  \dfrac{n\mathcal{H}_{\Sigma}}{n+1}\int_{\Sigma}{\mbox{div}(\xi)}d\Sigma=0.
	\end{eqnarray}
	
	Since $\mbox{Ric}_{\overline{M}}(\xi,N)\leq0$ and $\nu\geq0,$ Equation \eqref{equacaofundamental} gives
	\begin{equation}\label{eq12}
		0=\int_{\Sigma}\mbox{div}_{\Sigma}(A(\xi^{\top}))d\Sigma=\int_{\Sigma}-\mbox{Ric}_{\overline{M}}(\xi^{\top},N)d\Sigma+\int_{\Sigma}\nu||A||^{2}d\Sigma.
	\end{equation} 
	
	Hence, $\|A\|=0$ along $\Sigma$. In particular, $\mathcal{H}_{\Sigma}$ is zero. Thus $\mathcal{H}\leq 0.$ 
 	
 	Now, similarly, if $q$ is a minimal point of $\mathcal{H}$ in $\overline{M},$ we can prove that $\mathcal{H}\geq0.$ So, $\mathcal{H}$ is identically zero in $\overline{M},$ contradicting our assumption that $\mathcal{H}$ is not constant on $\overline{M}$. 
 	
 	{\it Case 1}: Supposed $L$ is a compact leaf  of $\mathcal{F}$. Then $\mbox{Ric}(\xi^{\top}, N)\leq 0 $ and $\xi^{\top}(n\mathcal{H})=0$ combined with Proposition $\ref{prop:eqfundamental}$ and Stoke's Theorem implies that
 	\begin{eqnarray}
 		0=\int_{L}\mbox{div}(A(\xi^{\top}))=\int_{L}{\left(-\mbox{Ric}_{\overline{M}}(\xi^{\top},N) +\nu||A||^{2}\right)dL}.\nonumber
 	\end{eqnarray} 
 This means that $\mbox{Ric}_{\overline{M}}(\xi^{\top},N)=0$ and $\|A\|=0$ along $L$.
 	
 	{\it Case 2}: Supposed that $L$ is a complete non-compact leaf. Since $\xi$ is a closed and conformal vector field, $$(n+1)\varphi|_{L}=\mbox{div}(\xi)|_L=\mbox{div}_L(\xi)+\varphi|_L,$$ 
 	so, $\varphi|_L=\frac{\mbox{div}_L(\xi)}{n}$. Then, we can write Equation $(\ref{equacaofundamental})$ as follows
 	\begin{eqnarray}\label{eq14}
 		\mbox{div}_{L}(A(\xi^{\top})-\mathcal{H}\xi^{\top}) & = &-\mbox{Ric}_{\overline{M}}(\xi^{\top},N) + \nu||A||^{2}.
 	\end{eqnarray}
 	Consider the vector field $X= A(\xi^{\top})-\mathcal{H}\xi^{\top}$. If this field $\|X\|\in \mathcal{L}^1(\mathcal{F})$ and the function $\mbox{div}_L(X)$ does not change sign, by Proposition $1$ in \cite{CamargoCS-2010}, we have that $\mbox{div}_{L}(X)=0.$ Consequently, $\mbox{Ric}_{\overline{M}}(\xi^{\top},N)=0$ and $\|A\|=0$ along of leaf  $L,$ i.e., the leaf  $L$ is totally geodesic.
 	
 	Finally, when $\nu\geq\dfrac{1}{||A||^{2}+1}$, Equation $\eqref{eq14}$ gives
 	\begin{equation}
 		0=\mbox{div}_{L}(A(\xi^{\top})-\mathcal{H}\xi^{\top})\geq \nu||A||^{2}\geq 1-\nu\geq 0.\nonumber
 	\end{equation}
 	So, $\nu=1$. Consequently, $\mathcal{F}$ is the Montiel foliation.
 	
 	{\it Case 3}: Supposed $L$ a complete non-compact leaf  of $\mathcal{F}$. Define $f=1-\overline{g}(N,\xi)$ on $\overline{M}$. Equation \eqref{eq9} allows us to conclude that $f\geq 0$ on $L$. Proving that $f\equiv 0$ on $L$ implies that $N$ will be identically equal to $\xi$, and the foliations are the same. Indeed, given $u, v\in T_pL$, we have
 	\begin{eqnarray*}
 		\overline{g}(A(u),v) =  -\overline{g} (\overline{\nabla}_uN,v)
 		=  -\overline{g}(\overline{\nabla}_u\xi,v)
 	 =  \overline{g}(-\varphi u,v),
 	\end{eqnarray*}
 	thus $A=-\varphi Id$. So $L$ is a totally umbilic submanifold. In addition, when $\nu=1$, then $\nu$ satisfies the inequality $\nu\geq\frac{1}{\|A\|^{2}+1}$ and consequently $L$ is totally geodesic. 
 	
 	Thus, it is sufficient to show that $f$ is identically zero along $L$. Consider $p\in L\in \mathcal{F}$, a vector $v\in T_pL$, then
 	\begin{eqnarray*}
 		\overline{g}(\nabla \nu,v)  =  \overline{g}(\overline{\nabla}_v N,\xi)+\overline{g}(N,\overline{\nabla}_{v}\xi)  =  \overline{g}(-A(\xi^{\top}),v).
 	\end{eqnarray*}
 	Hence, the gradient of $\nu$ is $\nabla \nu =-A(\xi^{\top})$, so $\overline{g}(\nabla f,X)\geq 0$.
 	From the Equation \eqref{desigualdadeNU}, we have
 	$$\mbox{div}_L(A\xi^{\top}-\mathcal{H}\xi^{\top}) =-\mbox{Ric}_{\overline{M}}(\xi^{\top},N)+\nu\|A\|^2.$$
 	The ambient space $\overline{M}$ have non-negative Ricci curvature, then 
 	$$\mbox{div}_L(X)\geq\nu \|A\|^2.$$
	Since $\nu\geq\dfrac{1}{||A||^{2}+1}$, we have $\mbox{div}_L(X)\geq f$ along $L$. Since the operator $A$ is bounded and the field $\xi$ is unitary, therefore the field $X$ is bounded. Since $L$ has polynomial volume growth, using the Theorem $2.1$ in \cite{AliasCN-2021}, we can conclude that $f\leq 0$ along $L$. And, by the inequality \eqref{eq9}, we have $f\geq 0$ on $L$, so $f\equiv 0$ on $L$. Therefore, $N=\xi$ along $L$, then $L$ is a totally geodesic leaf of $\mathcal{F}(\xi)$.
 	
\end{proof}

\subsection{Proof of Theorem C and its consequences}
%

%

\begin{proof}
First, we are going to prove that \(\mathcal{G}^{\xi}_L \geq 0\). Suppose that  $\mathcal{G}^{\xi}_{L}<0.$ Thus, from the Equation~\eqref{equacaofundamental} and the fact $n\mathcal{H}^{2}\leq\|A\|^{2}$, we have
	\begin{eqnarray}
		\label{eqsup2}
		\nu(n\mathcal{H}^{2})+{\xi}^{\top}(n\mathcal{H})\leq \nu\|A\|^{2}+{\xi}^{\top}(n\mathcal{H})<0.
	\end{eqnarray}
Let $\gamma$ be an integral curve of ${\xi}^{\top}$. Inequality~\eqref{eqsup2} can be restated as
\begin{eqnarray}\label{eqsupgamma}
	\nu(n\mathcal{H}_{\gamma(s)}^{2})+n{\mathcal{H}'}_{\gamma(s)}&<&0\nonumber\\
	{\mathcal{H}'}_{\gamma(s)}<-\nu \mathcal{H}_{\gamma(s)}^{2}&<&0.
\end{eqnarray}
The ambient space $\overline{M}^{n+1}$ is complete, which means that $\gamma$ can be extended to all $s\in\mathbb{R}$. We can choose the normal field $N$ such that $\mathcal{H}_{\gamma}\leq 0$. We just need to replace the fields $N$ and $\xi$ by the fields $-N$ and $-\xi$, since  $\mbox{div}_L(A({\xi}^{\top}))$, the field $\mathcal{H}{\xi}^{\top}$ and the curvature $\mbox{Ric}_{\overline{M}}({\xi}^{\top},N)$ do not change. Thus, consider the interval $I=]a,+\infty[$ such that $\mathcal{H}$ is negative on $I.$ In this case,
\begin{eqnarray}
	\mathcal{H}_{\gamma(s)}<0 \,\,\,\text{, }\,\,\,{\mathcal{H}'}_{\gamma(s)}<0\,\,\,\text{ e }\,\,\,\dfrac{-{\mathcal{H}'}_{\gamma(s)}}{\mathcal{H}^{2}_{\gamma(s)}}>\nu\geq0\nonumber,
\end{eqnarray} 
for $s>a.$ Define the auxiliary function $G(s)=-\dfrac{1}{\mathcal{H}_{\gamma(s)}}.$ From the Equation~\eqref{eqsupgamma}, we have
\begin{eqnarray}\label{eqG}
	{G'}(s)=\dfrac{{\mathcal{H}'}_{\gamma(s)}}{\mathcal{H}^{2}_{\gamma(s)}}<-\nu.
\end{eqnarray}
The Mean Value Theorem applied on $G$, for a fixed $b$, $b<s\in I$,
\begin{eqnarray}\label{eqABS}
	-\dfrac{1}{\mathcal{H}_{\gamma(s)}}+\dfrac{1}{\mathcal{H}_{\gamma(b)}}&=&\left(\dfrac{{\mathcal{H}'}_{\gamma(s)}}{\mathcal{H}^{2}_{\gamma(y)}}\right)(s-b)\nonumber\\
	-\dfrac{1}{\mathcal{H}_{\gamma(s)}}+\dfrac{1}{\mathcal{H}_{\gamma(b)}}&<&-\nu_{0}(s-b)
\end{eqnarray}
When $s$ grows to $+\infty$, the left hand side of the inequality \eqref{eqABS} converges to a constant while the right hand side tends to $-\infty,$ a contradiction. Therefore,  $\mathcal{G}_{L}^{\xi}\geq0$. 
 
 Now we prove inequality \eqref{eqGmaiorqueH}. From the previous discussion, $\mathcal{G}_{L}^{\xi}=a^{2},$ where $a$ is a non-negative real number. If $\mathcal{G}_{L}^{\xi}$ is infinite, there is nothing to prove. Assume $\mathcal{G}_{L}^{\xi}< \mathcal{H}_{L}^2$. Then, there exists $p\in L$ such that at this point
 	\begin{equation}\label{eqHabsurdo}
 		\mathcal{H}^{2}_{p}>a^{2}.
 	\end{equation}
We are going to consider $a>0$ and $a=0$ separately.   
 
 Assume $a>0.$ Inequality~\eqref{eqHabsurdo} provides $\mathcal{H}^{2}-a^{2}>0.$ Consider $\gamma$ the integral curve of $\xi^{T}$, with $\gamma(0)=p.$ We can choose a direction such that $\mathcal{H}_{p}<0.$ Thus, using \eqref{eqHabsurdo}, we conclude that $\mathcal{H}_{p}=\mathcal{H}_{\gamma(0)}<-a.$ Since $\mathcal{H}$ is a continuous function, there is a maximal interval $[0,b[$ where 
 \begin{equation}\label{eqHabsurdointervalo}
 	\mathcal{H}^{2}_{\gamma(s)}-a^{2}>0,\,\,\,\forall s\in [ 0,b[   .
 \end{equation}
We have that $b=+\infty$, otherwise, if $b$ is finite, we would have $\mathcal{H}_{\gamma(b)}=-a,$ that is, $\mathcal{H}_{\gamma(b)}^{2}=a^{2}.$ Thus, by the Equation \eqref{equacaofundamental} and the inequality $n\mathcal{H}\leq\|A\|^{2}$,
 \begin{eqnarray}\label{eqHgamma}
 	{\xi}^{\top}(n\mathcal{H}) &=& \mbox{div}_{L}\left( A({\xi}^{\top}-\mathcal{H}\xi^{\top})\right)+\mbox{Ric}_{\overline{M}}(\xi^{\top},N)-\nu\|A\|^{2}\nonumber\\
 	&\leq&\mathcal{G}_{L}^{\xi}-\nu\|A\|^{2}\leq a^{2}-n\nu \mathcal{H}^{2}\nonumber\\
 	&\leq& a^{2}-\mathcal{H}^{2}.\nonumber
 \end{eqnarray}
 This implies ${\mathcal{H}'}_{\gamma(s)}<0$ for all $s\in[0,b[,$ or $\mathcal{H}$ is strictly decreasing on $[0,b]$. On the other hand, $\mathcal{H}_{\gamma(0)}<-a$ and $\mathcal{H}_{\gamma(b)}=-a$, a contradiction. The initial statement is true, $b=+\infty.$ From this step, we can write the following inequalities, for all $s>0.$
 \begin{equation}\label{eqHabsurdototal}
 	\mathcal{H}_{\gamma(s)}<0,\,\,\,\,\,\,\, {\mathcal{H}'}_{\gamma(s)}\leq 0\,\,\,\,\, \text{  e  }\,\,\,\,\, \dfrac{{\mathcal{H}'}_{\gamma(s)}}{a^{2}-\mathcal{H}_{\gamma(s)}^{2}}\geq1.
 \end{equation}
When $s>0,$ let $h$ be the following function $$h(s)=\dfrac{1}{2a}\ln\left(\dfrac{\mathcal{H}_{\gamma(s)}-a}{\mathcal{H}_{\gamma(s)}+a}\right).$$ By fixing a strictly positive real number $b_{0}$ and applying the Mean Value Theorem for $h$ restricted to $[b_{0},s],$ there exists a number $c\in ]b_{0},s[$, such that $h(s)-h(b_{0})={h'}(c)(s-b_{0})$, i.e.,
 \begin{eqnarray}
 	\dfrac{1}{2a}\ln\left(\dfrac{\mathcal{H}_{\gamma(s)}-a}{\mathcal{H}_{\gamma(s)}+a}\right)-\dfrac{1}{2a}\ln\left(\dfrac{\mathcal{H}_{\gamma({b_{0}})}-a}{\mathcal{H}_{\gamma(b_{0})}+a}\right)<(-1)(s-b_{0}).\nonumber
 \end{eqnarray}
 
 When $s$ tends to $+\infty$ we have that the left hand side tends to a constant while the right hand side tends to $-\infty$, a contradiction. Therefore, there is no point $p\in L$ such that $\mathcal{H}_{p}<a^{2}.$
 
 The case $a=0$ follows a similar argument. Therefore, $\mathcal{G}^{\xi}_{L}\geq \mathcal{H}^{2}.$
 
 Suppose that $\mathcal{G}_{L}^{\xi}=0$. The inequality~\eqref{eqGmaiorqueH} implies $\mathcal{H}_{L}=0.$ Since $\mathcal{G}_{L}^{\xi}$ is the supreme on the leaf $L$, Equation~\eqref{equacaofundamental} provides
 \begin{equation*}
 	{\xi}^{\top}\left(n\mathcal{H}\right)+\nu\|A\|^{2} = \mbox{div}_{L}\left( A({\xi}^{\top})-\mathcal{H}{\xi}^{\top} \right)+\mbox{Ric}_{\overline{M}}\left( {\xi}^{\top},N \right) \leq\mathcal{G}_{L}^{\xi}=0,
 \end{equation*}
 and $\nu\|A\|^{2}=0$, i.e., $A=0$. Therefore, the leaf $L$ is totally geodesic.
 
 Conversely, suppose that $L\in \mathcal{F}$ is a totally geodesic leaf of $\mathcal{F}$. Then $A=0$ and $\mathcal{H}_L=0$. We can rewrite \eqref{equacaofundamental} along the leaf $L,$ as
 \begin{eqnarray*}
 	\mbox{div}_{L}\left(A({\xi}^{\top}) - \mathcal{H}{\xi}^{\top}\right)+\mbox{Ric}_{\overline{M}}({\xi}^{\top},N) = 0,
 \end{eqnarray*}
 meaning $\mathcal{G}_{L}^{\xi}=0.$
 \end{proof}

The following corollary highlights an important application of the maximum principal of Yau \cite{yau:1975} to the study of transverse foliations. 
\begin{corollary}
	Let $\overline{M}^{n+1}$ be a Riemannian manifold equipped with a Montiel Foliation $\mathcal{F}(\xi)$. Let $\mathcal{F}$ be a foliation transverse to $\mathcal{F}(\xi)$ and $L\in\mathcal{F}$ where $\nu\geq \nu_{0}>0$ and $\dim L\geq \frac{1}{\nu_{0}} $. If $L$ is a complete leaf with Ricci curvature bounded from below and $\mathcal{G}_{L}^{\xi}\in \mathbb{R}$, then there is a sequence of points $\{p_{k}\}\in L$ such that:
	\begin{enumerate}
		\item $\lim\limits_{k\rightarrow +\infty}{\mathcal{H}_{L}(p_{k})}=\sup_{L}{\mathcal{H}_{L}}$;
		\item $\lim\limits_{k\rightarrow +\infty}{\|\nabla \mathcal{H}_{L}(p_{k})\|}=0$;
		\item $\lim\limits_{k\rightarrow +\infty}{\Delta \mathcal{H}_{L}(p_{k})}\leq 0$.
	\end{enumerate} 
	
\end{corollary}
\begin{proof}	
	It follows from the Theorem~\ref{TeoInvariante} that $\mathcal{H}_{L}^{2}\leq \mathcal{G}_{L}^{\xi}$ and we have that $\mathcal{G}_{L}^{\xi}<+\infty$, so the curvature mean function $\mathcal{H}_{L}$ is bounded. By \cite{yau:1975}, we conclude the proof.
\end{proof}

Another consequence of the Theorem~\ref{TeoInvariante} combined with the Hopf–Calabi Theorem in \cite{Calabi-1957} is given below. This is Corollary D from the introduction.

\begin{corollary}
	Let $\overline{M}^{n+1}$ be a Riemannian manifold equipped with a Montiel Foliation $\mathcal{F}(\xi)$. Let $\mathcal{F}$ be a foliation transverse to $\mathcal{F}(\xi)$ and $L\in\mathcal{F}$ where $\nu\geq \nu_{0}>0$ and $\dim L\geq \frac{1}{\nu_{0}} $. If $\mathcal{H}_L$ is subharmonic and there is $p\in L$ such that $\mathcal{H}_L(p)=\left(\mathcal{G}_{L}^{\xi}\right)^{\frac{1}{2}}\in \mathbb{R}$, then $L$ has constant mean curvature. 
\end{corollary}

\begin{proof}
	 The application $\mathcal{H}_{L}$ is subharmonic and there is a $p\in L$ such that $\mathcal{H}_{L}(p)=\big(\mathcal{G}_{L}^{\xi}\big)^{\frac{1}{2}}.$ Using the principle of the Hopf-Calabi maximum in \cite{Calabi-1957}, it follows that $\mathcal{H}_{L}$ is constant.
\end{proof}

\noindent {\bf Acknowledgements.} The authors would like to thank Alexandre Fernandes and Edson Sampaio for their interest in this manuscript. In addition, we thank Fabiano Brito, Eduardo Longa and Francisco Gozzi for their precise suggestions that greatly improve this manuscript.

\end{document}